\documentclass{article}
\usepackage{tikz}
\usepackage{multirow}
\usepackage{pstricks,color,pst-node,pst-tree}
\usepackage{amssymb,latexsym}
\usepackage{amsfonts,amsmath}
\topskip  \usepackage{tikz}
\usepackage{multirow}
\usepackage{pstricks,color,pst-node,pst-tree}
\usepackage{amssymb,latexsym}
\usepackage{amsfonts,amsmath}
\topskip  \newenvironment{proof}[1][Proof]{\paragraph*{#1}}{\hspace*{\fill}$\Box$\bigskip}
\newtheorem{theorem}{Theorem}
\newtheorem{proposition}[theorem]{Proposition}
\newtheorem{corollary}[theorem]{Corollary}

\newtheorem{lemma}[theorem]{Lemma}

\def\sumrecord{{\rm sumrec}}
\def\sumweightedrecord{{\rm swrec}}
\def\sumelements{{\rm sume}}
\def\sumelementsa{{\rm sume_a}} 
\begin{document}
\title{Sum of elements preceding records in set partitions}

\author{ Walaa Asakly and Noor Kezil\\
Department of Mathematics, Braude college, Karmiel, Israel\\
Department of Mathematics, University of Haifa , Haifa, Israel\\
{\tt walaa\_asakly@hotmail.com}\\
{\tt nkizil02@campus.haifa.ac.il}}

\date{\small }%\today}
\maketitle

\begin{abstract}
In this paper, we aim to derive an explicit formula for the total number of elements preceding records over all set partitions of $[n]$ with exactly $k$ blocks, as well as an asymptotic estimate for the total sum of elements preceding records in all set partitions of $[n]$, expressed in terms of Bell numbers. To achieve this, we analyze the generating function that enumerates set partitions of $[n]$ according to this statistic, which we denote by $\sumelements$.

\medskip

\end{abstract}
%------------------------------------------------
\section{Introduction}
Define \( [k] = \{1, 2, \ldots, k\} \) as a totally ordered alphabet with \( k \) letters, and let \( [k]^n \) be the set  of words of length $n$  over this alphabet. Let
$\omega=\omega_1\omega_2\cdots\omega_n\in[k]^n$,  an element
$\omega_i$ in $\omega$ is a {\em record} if $\omega_i>\omega_j$
for all $j=1, 2, \cdots, i-1$, and $i$ is called the \emph{position} of the record $\omega_i$. Researchers have studied the record statistics on various combinatorial structures. The statistic {\em $\sumrecord$} was defined and studied by Kortchemski \cite{I} 
on permutations (words without repeated letters). For a permutation $\pi$, 
$\sumrecord(\pi)$ is defined as the sum of the positions of all records in $\pi$. For example, the permutation $\pi=122634$
has $3$ records,  $1$, $2$ and $6$, occur at positions $1$, $2$ and $4$, respectively, yielding 
$\sumrecord(\pi)=1+2+4=7$. In the context of compositions  (a composition
$\sigma=\sigma_1\sigma_2\cdots\sigma_m$ of an integer $n$ is a
word over alphabet $\mathbb{N}$ whose sum is $n$), Knopfmacher and Mansour \cite{AT} studied the total number of records in all compositions of a given integer $n$, as well as the sum of the positions of these records.  Myers and Wilf \cite{AH}
studied records on words over finite alphabet. 

 A {\em set partition} $\Pi$ of
$[n]$ of size $k$ (or, {\em set partition of $[n]$ with exactly
$k$ blocks}) is a collection $\{B_1,B_2,\ldots,B_k\}$ of nonempty
disjoint subsets of $[n]$, called {\em blocks}, whose union is
equal to $[n]$. We assume that blocks are listed in increasing
order of their minimal elements, that is, $minB_1 < minB_2 <
\cdots< min B_k$. We denote the set of all set partitions of $[n]$
with exactly $k$ blocks to be $P_{n,k}$, and we denote the set of
all set partitions of $[n]$ to be $P_n$. It is well-known that the
number of all set partitions of $[n]$  with exactly $k$  blocks is given by the Stirling 
numbers of the second kind $S_{n,k}$, and the
number of all set partitions of $[n]$ is the $n$-th Bell
number $B_n$, see \cite{TM}.
A partition $\Pi$ can be written as a word
$\pi=\pi_1\pi_2\cdots\pi_n$, where $i\in B_{\pi_i}$ for all $i$,
and this form is called the {\em canonical sequential form}, see \cite{TM}. For example, the partition $\Pi=\{\{1,5\},\{2,3\},\{4\}\}\in P_{5,3}$ has the canonical sequential form 
$\pi=12231$.  Knopfmacher,  Mansour,  and Wagner  \cite{ATS} found
the asymptotic mean values and variances for the number, and for the sum of positions of records in all partitions of $[n]$. Moreover, Asakly \cite{AW} studied the statistic \emph{swrec}, where $swrec(\pi)$  is defined as the sum of the position of a record in $\pi$ multiplied by the value of
the record over all the records in $P_n$.  For instance, if $\pi= 122313$, then $swrec(\pi)=1\cdot1+2\cdot2+3\cdot4=17$.

Motivated by the aforementioned results, we define $\sumelementsa(\pi)$ to be the sum of all elements preceding a
record  $a$ in $\pi$, and define 
$\sumelements(\pi)$ as the total sum of all such $\sumelementsa(\pi)$  over all records in $\pi$. In this paper, we study this statistic in set partitions in context of canonical sequential form. For instance, if $\pi=121132$, then $\sumelements_2(\pi)=1$, 
$\sumelements_3(\pi)=1+2+1+1$ and $\sumelements(\pi)=1+(1+2+1+1)=6$. In
particular, we show that the total number of $\sumelements(\pi)$
taken over all set partitions of $[n]$ is given by
$$\frac{1}{3}B_{n+3}-\frac{1}{4}B_{n+2}-(\frac{1}{2}n+\frac{13}{12})B_{n+1}-(\frac{1}{12}+\frac{1}{2}n)B_n,$$
see Theorem \ref{th}.

\section {Main Results}
\subsection{The ordinary generating function for the number of set partitions according to the statistic $\sumelements$}
Let $P_{k,a}(x,q)$ be the generating function for the number of set partitions of
$[n]$ with exactly $k$ blocks according to the statistic $\sumelements_a$, that is
$$P_{k,a}(x,q)=\sum_{n\geq k}\sum_{\pi\in P_{n,k}}x^nq^{\sumelements_a(\pi)}.$$
And Let $P_k(x,q)$ be the generating function for the number of set partitions of
$[n]$ with exactly $k$ blocks according to the statistic $\sumelements$, that is
$$P_k(x,q)=\sum_{n\geq k}\sum_{\pi\in P_{n,k}}x^nq^{\sumelements(\pi)}.$$
We aim to find an explicit formula for the ordinary generating function $P_k(x,q)$, and using the
same methods as in \cite{ATS}, also obtain an explicit formula for the corresponding exponential generating function.

\begin{theorem}\label{th1}
The generating function for the number of set partitions of $[n]$ with exactly $k$ blocks
according to the statistic $\sumelements_a$ is given by

\begin{align}\label{eq1}
P_{k,a}(x,q)=x^kq^{\frac{a(a-1)}{2}}\left({\prod_{j=1}^{a-1}\left(\frac{1}{1-xq(\frac{1-q^j}{1-q})}\prod_{i=a}^{k}\frac{1}{1-ax}\right)}\right).
\end{align}
\end{theorem}

\begin{proof}
Let $\Pi$ be a set partition of $[n]$ with exactly $k$ blocks, and let $\pi$
denote its canonical sequential form, defined as follows:
$\pi=1\pi^{(1)}2\pi^{(2)}\cdots k\pi^{(k)}$, where $\pi^{(j)}$
denotes an arbitrary word over an alphabet $[j]$ including the
empty word. Thus, for a fixed $1\leq a\leq k-1$ the contribution of
$\pi'=1\pi^{(1)}2\pi^{(2)}\cdots (a-1)\pi^{(a-1)}$ to the
generating function $P_k(x,q)$ is $q^{1+2+\cdots +a-1}\prod_{j=1}^{a-1}\frac{x}{1-xq(\frac{1-q^j}{1-q})}$ and the
contribution of $a\pi^{(a+1)}\cdots (k-1)\pi^{(k-1)}k\pi^{(k)}$ to the generating function $P_k(x,q)$
is $\prod_{i=a}^{k}\frac{x}{1-ax}$,
as required.
\end{proof}
\begin{corollary}\label{c1}
The generating function for the number of set partitions of $[n]$ with exactly $k$ blocks
according to the statistic $\sumelements$ is given by

\begin{align}\label{eq2}
P_k(x,q)=\sum_{a=1}^{k}x^kq^{\frac{a(a-1)}{2}}\left({\prod_{j=1}^{a-1}\left(\frac{1}{1-xq(\frac{1-q^j}{1-q})}\prod_{i=a}^{k}\frac{1}{1-ax}\right)}\right).
\end{align}
\end{corollary}

\begin{proof}
By summing over all $1\leq a\leq k$ in (\ref{eq1}) , we obtain the required result.
\end{proof}

\subsection{Exact expression for $\sum_{\pi\in P_{n,k}}{\sumelements(\pi)}$}
In this section we aim to prove that the total number of the
$\sumelements$ over all set partitions of $[n]$ with exactly $k$ blocks is
$$S_{n,k}\sum_{a=1}^{k}\frac{a(a-1)}{2}+\sum_{i=1}^{k-1}\left(\frac{(k-i)i(i+1)}{2}\sum_{j=1}^{n-k}S_{n-j,k}i^{j-1}\right).$$

For that we need the following Lemma:

\begin{lemma}\label{L}
For all $k\geq
1$,
\begin{align}\label{eq3}
\frac{d}{dq}P_k(x,q)\mid_{q=1}=\frac{x^{k}}{(1-x)\ldots (1-kx)}\sum_{a=1}^{k}\frac{a(a-1)}{2}+\frac{x^{k+1}}{(1-x)\ldots (1-kx)}\sum_{i=1}^{k-1}{\frac{(k-i)i(i+1)}{2(1-ix)}}.
\end{align}
\end{lemma}
\begin{proof}
By differentiating (\ref{eq2}) with respect to $q$, we obtain
\begin{equation}\label{eq4}
\frac{d}{dq}P_k(x,q)\mid_{q=1}=x^k\sum_{a=1}^{k}{\lim_{q \rightarrow 1}\left(\frac{d}{dq}L_{j,a}(q)\prod_{i=a}^{k}\frac{1}{1-ix}\right)},
\end{equation}
where
$L_{j,a}(q)=q^{\frac{a(a-1)}{2}}\prod_{j=1}^{a-1}\frac{1}{1-xq(\frac{1-q^j}{1-q})}$.
Observe that
\begin{align*}
&\frac{d}{dq}L_{j,a}(q)=\frac{a(a-1)}{2}q^{\frac{a(a-1)}{2}-1}\prod_{j=1}^{a-1}\frac{1}{1-xq(\frac{1-q^j}{1-q})}\\
&+q^{\frac{a(a-1)}{2}}\sum_{m=1}^{a-1}\left(\frac{x(\frac{1-q^j}{1-q})
+xq(\frac{-jq^{j-1}(1-q)+(1-q^j)}{(1-q)^2})}{(1-xq(\frac{1-q^j}{1-q}))^2}\right)\left(\prod_{j=1, j\neq m}^{a-1}\frac{1}{1-xq(\frac{1-q^j}{1-q})}\right).
\end{align*}
 By substituting the following facts
  \begin{align*}
&\lim_{q\rightarrow 1}x(\frac{1-q^j}{1-q})=xj,\\
&\lim_{q\rightarrow 1}(1-xq(\frac{1-q^j}{1-q}))^2=(1-xj)^2,\\
&\lim_{q\rightarrow 1}xq(\frac{-jq^{j-1}(1-q)+(1-q^j)}{(1-q)^2})=\frac{xj(j-1)}{2},
\end{align*}
 into (\ref{eq4}), we obtain the required result.
\end{proof}

\begin{theorem}
The total number of the $\sumelements$ over all set partitions of $[n]$ with exactly $k$ blocks is
$$S_{n,k}\sum_{a=1}^{k}\frac{a(a-1)}{2}+\sum_{i=1}^{k-1}\left(\frac{(k-i)i(i+1)}{2}\sum_{j=1}^{n-k}S_{n-j,k}i^{j-1}\right).$$
\end{theorem}
\begin{proof}
 In order to enumerate the total number of all partitions of $[n]$ with exactly k blocks
 according to  the $\sumelements$ statistic , we need to find the coefficients of $x^n$ in (\ref{eq3}). We start with the fact that
\begin{align*}
\frac{x^k}{\prod_{j=1}^{k}(1-jx)}=\sum_{n\geq k}S_{n,k}x^n,
\end{align*}
where $S_{n,k}$ denotes the Stirling numbers of the second kind.

Moreover, for any integer $1\leq i\leq k-1$, we have
\begin{align*}
\frac{x^{k+1}}{\prod_{j=1}^{k}(1-jx)}\cdot \frac{1}{1-ix}\cdot \frac{(k-i)i(i+1)}{2}=\frac{(k-i)i(i+1)}{2} \sum_{n\geq k}S_{n,k}x^{n} \sum_{n\geq1}i^{n-1}x^n
\end{align*}
which is equal to
\begin{align*}
\frac{(k-i)i(i+1)}{2}\sum_{n\geq 0}\sum_{j=1}^{n-k}S_{n-j,k}i^{j-1}x^{n}.
\end{align*}
Which completes the proof.
\end{proof}
\subsection{Exact expression for   $\sum_{\pi\in P_{n}}{\sumelements(\pi)}$}
In this section we aim to prove that the total number of the
$\sumelements$ over all set partitions of $[n]$ is
$$\frac{1}{3}B_{n+3}-\frac{1}{4}B_{n+2}-(\frac{1}{2}n+\frac{13}{12})B_{n+1}-(\frac{1}{12}+\frac{1}{2}n)B_n.$$

For that we pass from  $\frac{d}{dq}P_k(x,q)\mid_{q=1}$
 to $\frac{d}{dq}\widetilde{P}(x,u,q)\mid_{u=q=1}$,
 where $\widetilde{P}(x,u,q)$ is the exponential generating function
for the number of set partitions of $[n]$ with exactly $k$ blocks
according to the statistic $\sumelements$, where $x,u,q$
mark $n,k,\sumelements$. We derive the total number of $\sumelements$
over all set partitions of $[n]$, by finding the coefficients $x^n$ in $\frac{d}{dq}\widetilde{P}(x,u,q)\mid_{u=q=1}$ .
Define $\widetilde{P}(x,u,q)$ by
$$\widetilde{P}(x,u,q)=\sum_{k\geq 0}\sum_{n\geq k}
\sum_{\pi\in P_{n,k}}\frac{x^nu^kq^{\sumelements(\pi)}}{n!}.$$

In order to find $\widetilde{P}(x,u,q)$, we need the following
result.
\begin{proposition}\label{LL}
The partial fraction decomposition of
$\frac{d}{dq}P_k(1/y,q)\mid_{q=1}$ can be expressed as
\begin{align}\label{eqq}
\sum_{m=1}^{k}\left(\frac{a_{k,m}}{(y-m)^2}+\frac{b_{k,m}}{y-m}\right),
\end{align}
 where
\begin{align*}
a_{k,m}&=\frac{(-1)^{k-m}(k-m)m(m+1)}{2(m-1)!(k-m)!},\\
b_{k,m}&=\frac{(-1)^{k-m}\left(\frac{k^3}{12}-\frac{k^2(m+1)}{4}+\frac{k(6m^2+21m+10)}{12}-\frac{3m^2}{2}-m +\sum_{i=1}^{k}\frac{i(i-1)}{2}\right)}{(m-1)!(k-m)!}.
\end{align*}
\end{proposition}
\begin{proof}
Equation \eqref{eq2} can be written as
\begin{align*}
\frac{d}{dq}P_k(x,q)\mid_{q=1}&=x^k\left(\sum_{a=1}^{k}\frac{a(a-1)}{2}+\sum_{i=1}^{k-1}\frac{(k-i)i(i+1)x}{2(1-ix)}\right)\prod_{i=1}^{k}\frac{1}{1-ix}.
\end{align*}
By substituting $x^{-1}=y$ in the above equation, we obtain
\begin{align}\label{eq7}
\frac{d}{dq}P_k(1/y,q)\mid_{q=1}&=\left(\sum_{a=1}^{k}\frac{a(a-1)}{2}+\sum_{i=1}^{k}\frac{(k-i)i(i+1)}{2(y-i)}\right)\prod_{i=1}^{k}\frac{1}{y-i}.
\end{align}
The expression above can be decomposed in the form
$$\frac{d}{dq}P_k(1/y,q)\mid_{q=1}=\sum_{m=1}^{k}\left(\frac{a_{k,m}}{(y-m)^2}+\frac{b_{k,m}}{y-m}\right).$$
To determine the coefficients $a_{k,m}$ and $b_{k,m}$, we
consider Laurent expansion of \eqref{eq4} around $y=m$:
\begin{align*}
\frac{d}{dq}P_k(1/y,q)\mid_{q=1}&=\frac{\sum_{a=1}^{k}\frac{a(a-1)}{2}+\frac{(k-m)m(m+1)}{2(y-m)}+\sum_{\substack{i=1 \\ i\neq m}}^{k}\frac{(k-i)i(i+1)}{2(y-i)}}{(y-m)\prod_{\substack{i=1 \\ i\neq m}}^{k}(y-m+m-i)}\\
&=\frac{\sum_{a=1}^{k}\frac{a(a-1)}{2}+\frac{(k-m)m(m+1)}{2(y-m)}+\sum_{\substack{i=1
\\ i\neq m}}^{k}\frac{(k-i)i(i+1)}{2(y-i)}}{(y-m)\prod_{\substack{i=1 \\ i\neq
m}}^{k}\left((m-i)(1+\frac{y-m}{m-i})\right)}.
\end{align*}
By using Laurent expansion of $\frac{1}{(1+\frac{y-m}{m-i})}$ and
$\frac{(k-i)i(i+1)}{2(y-i)}$ with $i\neq m$ around $y=m$, we obtain
\begin{align*}
&\frac{d}{dq}P_k(1/y,q)\mid_{q=1}\\
&=\frac{(-1)^{k-m}}{(y-m)(m-1)!(k-m)!}\prod_{\substack{i=1 \\ i\neq m}}^{k}\left(1-\frac{y-m}{m-i}+O((y-m)^2)\right)\\
&\qquad\cdot\left(\sum_{a=1}^{k}\frac{a(a-1)}{2}+\frac{(k-m)m(m+1)}{2(y-m)}+\sum_{\substack{i=1
\\ i\neq m}}^{k}\frac{(k-i)i(i+1)}{2(m-i)}+O(y-m)\right),
\end{align*}
which equals %%%carry out the expansion to a third term as follows:
\begin{align*}
&\frac{d}{dq}P_k(1/y,q)\mid_{q=1}\\
&=\frac{(-1)^{k-m}}{(y-m)(m-1)!(k-m)!}\left(1-\sum_{\substack{i=1 \\ i\neq m}}^{k}\frac{y-m}{m-i}+O((y-m)^2)\right)\\
&\qquad\cdot\left(\sum_{a=1}^{k}\frac{a(a-1)}{2}+\frac{(k-m)m(m+1)}{2(y-m)}+\sum_{\substack{i=1
\\ i\neq m}}^{k}\frac{(k-i)i(i+1)}{2(m-i)}+O(y-m)\right).
\end{align*}

We need to simplify the product, and consider the coefficients of $(y-m)^-1$ and $(y-m)^-2$ as follows:
\begin{align*}
&\frac{(-1)^{k-m}}{(y-m)(m-1)!(k-m)!}\\
&\qquad\cdot\left(\sum_{a=1}^{k}\frac{a(a-1)}{2}+\frac{(k-m)m(m+1)}{2(y-m)}+\sum_{\substack{i=1
\\ i\neq m}}^{k}\frac{(k-i)i(i+1)-(k-m)m(m+1)}{2(m-i)}+O(y-m)\right).
\end{align*}
With Maple, we can compute the summand, which shows that
\begin{align*}
&\frac{(-1)^{k-m}}{(y-m)(m-1)!(k-m)!}\\
&\qquad\cdot\left(\frac{m(k-m)(m+1)}{2(y-m)}+\sum_{a=1}^{k}\frac{a(a-1)}{2}+\frac{1}{2}\sum_{\substack{i=1
\\ i\neq m}}^{k}i^2+(m-k+1)i+(m+1)(k-m)+O(y-m) \right)\\
&=\frac{(-1)^{k-m}}{(y-m)(m-1)!(k-m)!}\\
&\qquad\cdot\left(\frac{m(k-m)(m+1)}{2(y-m)}-\frac{k^3}{12}-\frac{k^2(m+1)}{4}+\frac{k(6m^2+21m+10)}{12}-\frac{3m^2}{2}-m+\sum_{a=1}^{k}\frac{a(a-1)}{2} \right).\\
\end{align*}
Consequently, finding the coefficients of $\frac{1}{(y-m)}$ and
$\frac{1}{(y-m)^{2}}$ completes the proof.
\end{proof}

Now, we proceed to find an explicit formula for the generating
function $\frac{d}{d
q}\widetilde{P}(x,1,q)\mid_{q=1}$.

\begin{theorem}\label{th2}
We have
\begin{align}
\frac{d}{d
q}\widetilde{P}(x,u,q)\mid_{u=q=1}=e^{e^x-1}\left(\frac{1}{3}e^{3x}-\frac{1}{2}xe^{2x}+\frac{3}{4}e^{2x}-xe^x-e^{x}-\frac{1}{12}\right).
\end{align}
\end{theorem}
\begin{proof}

By replacing $\frac{1}{y-m}=\frac{x}{1-mx}=\sum_{\ell\geq 0}m^\ell x^{\ell+1}$ with $\frac{e^{mx}-1}{m}=\sum_{\ell\geq0}\frac{m^\ell
x^{\ell+1}}{(\ell+1)!}$ and $\frac{1}{(y-m)^2}$
 with $\frac{e^{mx}(mx-1)+1}{m^2}$ in (\ref{eqq}), we pass to exponential generating function. Moreover, by multiplying the result by $u^k$ and summing over
all $k$, we obtain
\begin{align*}
\frac{d}{d
q}\widetilde{P}(x,u,q)\mid_{q=1}&=\sum_{k\geq 1}u^k{\sum_{m=1}^{k}\frac{(-1)^{k-m}(k-m)m(m+1)}{2(m-1)!(k-m)!}\cdot \frac{e^{mx}(mx-1)+1}{m^2}}\\
&+\sum_{k\geq
1}u^k{\sum_{m=1}^{k}\frac{(-1)^{k-m}\left(-\frac{k^3}{12}-\frac{k^2(m+1)}{4}+\frac{k(6m^2+21m+10)}{12}-\frac{3m^2}{2}-m +\sum_{i=1}^{k}\frac{i(i-1)}{2}\right)}{(m-1)!(k-m)!}\cdot
\frac{e^{mx}-1}{m}}.
\end{align*}
By changing the order of the summation, we get
\begin{align*}
\frac{d}{d
q}\widetilde{P}(x,u,q)\mid_{q=1}&=\sum_{m\geq1}\frac{e^{mx}(mx-1)+1}{m!}\sum_{k\geq m}\frac{(-1)^{k-m}(k-m)(m+1)u^k}{2(k-m)!}\\
&+\sum_{m\geq1}\frac{e^{mx}-1}{m!}\sum_{k\geq
m}\frac{(-1)^{k-m}\left(-\frac{k^3}{12}-\frac{k^2(m+1)}{4}+\frac{k(6m^2+21m+10)}{12}-\frac{3m^2}{2}-m +\sum_{i=1}^{k}\frac{i(i-1)}{2}\right)}{(k-m)!}u^k.
\end{align*}
By Substituting $\ell=k-m$ and rewriting the above equation, we obtain
\begin{align*}
\frac{d}{d
q}\widetilde{P}(x,u,q)\mid_{q=1}&=\sum_{m\geq1}\frac{e^{mx}(mx-1)+1}{m!}\sum_{\ell\geq 0}\frac{(-1)^{\ell}(m+1)\ell u^{m+\ell}}{2\ell!}\\
&+\sum_{m\geq1}\frac{e^{mx}-1}{m!}\sum_{\ell \geq
0}\frac{(-1)^{\ell}\left(-\frac{(\ell +m)^3}{12}-\frac{(\ell +m)^2(m+1)}{4}\right)}{\ell!}u^{m+\ell}\\
&+\sum_{m\geq1}\frac{e^{mx}-1}{m!}\sum_{\ell \geq
0}\frac{(-1)^{\ell}\left(\frac{(\ell+m)(6m^2+21m+10)}{12}-\frac{3m^2}{2}-m +\sum_{i=1}^{\ell +m}\frac{i(i-1)}{2}\right)}{\ell!}u^{m+\ell}.
\end{align*}
By substituting $u=1$ into the previous terms, we complete the proof.
\end{proof}

\begin{theorem}\label{th}
The total number of $\sumelements$ taken over all set
partitions of $[n]$ is given by
$$\frac{1}{3}B_{n+3}-\frac{1}{4}B_{n+2}-(\frac{1}{2}n+\frac{13}{12})B_{n+1}-(\frac{1}{12}+\frac{1}{2}n)B_n.$$
\end{theorem}
\begin{proof}
According to Theorem \ref{th2}, we have
$$\frac{d}{dq}\widetilde{P}(x,u,q)\mid_{u=q=1}=e^{e^x-1}\left(\frac{1}{3}e^{3x}-\frac{1}{2}xe^{2x}+\frac{3}{4}e^{2x}-xe^x-e^{x}-\frac{1}{12}\right).$$
Differentiating the generating function $$e^{e^x-1}=\sum_{n\geq0}B_n\frac{x^n}{n!}$$
three times, we get
$$e^xe^{e^x-1}=\sum_{n\geq0}B_{n+1}\frac{x^n}{n!},$$
$$e^{2x}e^{e^x-1}=\sum_{n\geq0}B_{n+2}\frac{x^n}{n!}-\sum_{n\geq0}B_{n+1}\frac{x^n}{n!},$$ and
$$e^{3x}e^{e^x-1}=\sum_{n\geq0}B_{n+3}\frac{x^n}{n!}-3\sum_{n\geq0}B_{n+2}\frac{x^n}{n!}+2\sum_{n\geq0}B_{n+1}\frac{x^n}{n!}.$$
From the above equations, it follows that
$$xe^xe^{e^x-1}=\sum_{n\geq0}nB_n\frac{x^n}{n!},$$
and
$$xe^{2x}e^{e^x-1}=\sum_{n\geq0}nB_{n+1}\frac{x^n}{n!}-\sum_{n\geq0}nB_n\frac{x^n}{n!}.$$
Using all these facts together, we obtain
\begin{align*}
&\frac{d}{dq}\widetilde{P}(x,u,q)\mid_{u=q=1}=\sum_{n\geq0}\left(\frac{1}{3}B_{n+3}-\frac{1}{4}B_{n+2}-(\frac{1}{2}n+\frac{13}{12})B_{n+1}-(\frac{1}{12}+\frac{1}{2}n)B_n\right)\frac{x^n}{n!},
\end{align*}
which completes the proof.
\end{proof}

In order to obtain asymptotic estimate for the moment as well as limiting distribution,
we need the fact
$$B_{n+h}=B_n\frac{(n+h)!}{n!r^h}\left(1+O(\frac{\log n}{n})\right)$$
uniformly for $h=O(\log n)$, where $r$ is the positive root of $re^r=n+1$.
 For more details about
the asymptotic expansion of Bell numbers, we refer the reader to
\cite{ER}. Therefore,
 Theorem \ref{th} gives the following corollary.

\begin{corollary}
Asymptotically, as $n\rightarrow\infty$, the total number of
$\sumweightedrecord$ taken over all set partitions of $[n]$ is
given by $B_n\frac{n^3}{r^3}\left(1+\frac{r}{n}\right)\left(1+O(\frac{\log n}{n})\right)$, where $r$
is the positive root of $re^r=n+1$.
\end{corollary}
In addition, we show that asymptotically the total number of the
$\sumelements$ over all set partitions of $[n]$ is given by
$$B_n\frac{n^3}{r^3}\left(1+\frac{r}{n}\right)\left(1+O(\frac{\log n}{n})\right),$$
where $r$ is the positive root of $re^r=n+1$.

\section*{Acknowledgments}
Walaa Asakly would like to dedicate this work to the memory of her father, 
whose love and encouragement continue to inspire her.

\end{document}